\documentclass{amsart}
\usepackage{amscd, amssymb, amsfonts}
\usepackage{amsmath,amsthm}
\usepackage{amssymb}

  [2002/01/19 v2.2g %
    AMS font definitions%
  ]
\DeclareFontFamily{U}{euf}{}
\DeclareFontShape{U}{euf}{m}{n}{%
  <5><6><7><8><9>gen*eufm%
  <10><10.95><12><14.4><17.28><20.74><24.88>eufm10%
  }{}
\DeclareFontShape{U}{euf}{b}{n}{%
  <5><6><7><8><9>gen*eufb%
  <10><10.95><12><14.4><17.28><20.74><24.88>eufb10%
  }{}

  [2002/01/19 v2.2g %
    AMS font definitions%
  ]
\DeclareFontFamily{U}{msb}{}
\DeclareFontShape{U}{msb}{m}{n}{%
  <5><6><7><8><9>gen*msbm%
  <10><10.95><12><14.4><17.28><20.74><24.88>msbm10%
  }{}

  [2002/01/19 v2.2g %
    AMS font definitions%
  ]
\DeclareFontFamily{U}{msa}{}
\DeclareFontShape{U}{msa}{m}{n}{%
  <5><6><7><8><9>gen*msam%
  <10><10.95><12><14.4><17.28><20.74><24.88>msam10%
  }{}

\newtheorem{theorem}{Theorem}[section]
\newtheorem{lemma}[theorem]{Lemma}
\newtheorem{proposition}[theorem]{Proposition}
\newtheorem{corollary}[theorem]{Corollary}

\theoremstyle{definition}

\theoremstyle{remark}

\numberwithin{equation}{section}

\begin{document}

\title[]
{On Carlitz's type $q$-Euler numbers associated with the fermionic $p$-adic integral on $\mathbb Z_p$}

\author{Min-Soo Kim, Taekyun Kim and Cheon-Seoung Ryoo}

\begin{abstract}
In this paper, we consider the following problem in \cite{KKL}:
``Find Witt's formula for Carlitz's type $q$-Euler numbers.''
We give Witt's formula for Carlitz's type $q$-Euler numbers,
which is an answer to the above problem.
Moreover, we obtain a new $p$-adic $q$-$l$-function $l_{p,q}(s,\chi)$ for Dirchlet's character $\chi,$
with the property that
$$l_{p,q}(-n,\chi)=E_{n,\chi_n,q}-\chi_n(p)[p]_q^nE_{n,\chi_n,q^p},\quad n=0,1,\ldots$$
using the fermionic $p$-adic integral on $\mathbb Z_p.$
\end{abstract}

\address{Department of Mathematics, KAIST, 373-1 Guseong-dong, Yuseong-gu, Daejeon 305-701, South Korea}
%   Current address
\email{minsookim@kaist.ac.kr}

\address{Division of General Education-Mathematics, Kwangwoon University, Seoul, 139-701, South Korea}
\email{tkkim@kw.ac.kr}

\address{Department of Mathematics, Hannam University, Daejeon 306-791, South Korea}
\email{ryoocs@hnu.kr}


%\thanks{Received May 24, 2009}

\maketitle

%%% --------------------------------------------------------

\maketitle


\def\ord{\text{ord}_p}
\def\C{\mathbb C_p}
\def\BZ{\mathbb Z}
\def\Z{\mathbb Z_p}
\def\Q{\mathbb Q_p}

\section{Introduction}

Throughout this paper, let $p$ be an odd prime number.
The symbol $\mathbb Z_p,\mathbb Q_p$ and $\mathbb C_p$ denote
the rings of $p$-adic integers, the field of $p$-adic numbers and
the field of $p$-adic completion of the algebraic closure of $\mathbb Q_p,$ respectively.
The $p$-adic absolute value in $\mathbb C_p$ is normalized in such way that
$|p|_p=p^{-1}.$
Let $\mathbb N$ be the natural numbers and $\mathbb Z^+=\mathbb N\cup\{0\}.$

As the definition of $q$-number, we use the following notations:
$$[x]_q=\frac{1-q^x}{1-q}\quad\text{and}\quad [x]_{-q}=\frac{1-(-q)^x}{1+q}.$$
Note that $\lim_{q\to1}[x]_q=x$ for $x\in\mathbb Z_p,$ where $q$ tends to 1 in the region $0<|q-1|_p<1.$

When one talks of $q$-analogue, $q$ is
variously considered as an indeterminate, a complex number $q\in\mathbb
C,$ or a $p$-adic number $q\in\mathbb C_p.$ If $q=1+t\in\mathbb C_p,$ one normally assumes $|t|_p<1.$
We shall further suppose that $\ord(t)>1/(p-1),$ so that $q^x=\exp(x\log q)$ for $|x|_p\leq1.$ If $q\in\mathbb C,$ then we assume that
$|q|<1.$

After Carlitz \cite{Ca48,Ca54} gave $q$-extensions of the classical Bernoulli numbers and polynomials,
the $q$-extensions of Bernoulli and Euler numbers and polynomials have been studied by
several authors (cf. \cite{CCKS,CKOS,Ca48,Ca54,KT1,KT,KT2,KT3,TK07,K07,K08,K08-1,K09,K09-1,K09-2,TK10,
KJKR,KKL,OS,OSRC,Sa}).
The Euler numbers and polynomials have been studied by researchers in the field of number theory, mathematical physics and
so on (cf. \cite{Ca48,Ca54,CSK,TK07,K08,K09,K09-1,K09-2,TK10,Ro}).
Recently, various $q$-extensions of these numbers and polynomials have been studied by many mathematicians
(cf. \cite{KT,KT2,KT3,K07,K08-1,KJKR,KKL,OSRC}).
Also, some authors have studied in the several area of $q$-theory (cf. \cite{CCKS,CKOS,GG,TK10,OS}).

It is known that the generating function of Euler numbers $F(t)$ is given by
\begin{equation}\label{Eu-gen}
F(t)=\frac 2{e^t+1}=\sum_{n=0}^\infty E_n\frac{t^n}{n!}.
\end{equation}
From (\ref{Eu-gen}), we known that the recurrence formula of Euler numbers is given by
\begin{equation}\label{Eu-gen-recu}
E_0=1,\qquad (E+1)^n+E_{n}=0\quad\text{if } n>0
\end{equation}
with the usual convention of replacing $E^n$ by $E_n$ (see \cite{KT2,KKL}).

In \cite{KJKR}, the $q$-extension of Euler numbers $E_{n,q}^*$ are defined as
\begin{equation}\label{Eu*-recu}
E_{0,q}^*=1,\qquad
(qE^*+1)^n+E_{n,q}^*=
\begin{cases}
2&\text{if }n=0 \\
0&\text{if } n>0
\end{cases}
\end{equation}
with the usual convention of replacing $(E^*)^n$ by $E_{n,q}^*.$

As the same motivation of the construction in \cite{KKL}, Carlitz's type $q$-Euler numbers $E_{n,q}$ are defined as
\begin{equation}\label{Eu-recu}
E_{0,q}=\frac{2}{[2]_q},\qquad
q(qE+1)^n+E_{n,q}=\begin{cases}
2&\text{if }n=0 \\
0&\text{if } n>0
\end{cases}
\end{equation}
with the usual convention of replacing $E^n$ by $E_{n,q}.$
It was shown that $\lim_{q\to1}E_{n,q}=E_n,$ where $E_n$ is the $n$th Euler number.
In the complex case, the generating function of Carlitz's type $q$-Euler numbers $F_q(t)$ is given by
\begin{equation}\label{Ca-q-Eu-int}
F_{q}(t)=\sum_{n=0}^\infty E_{n,q}\frac{t^n}{n!}=2\sum_{n=0}^\infty(-q)^ne^{[n]_qt},
\end{equation}
where $q$ is complex number with $|q|<1$ (see \cite{KKL}).
The remark point is that the series on the right-hand side of (\ref{Ca-q-Eu-int}) is uniformly convergent in the wider sense.
In $p$-adic case, Kim et al. \cite{KKL} could not determine the generating function of
Carlitz's type $q$-Euler numbers and Witt's formula for Carlitz's type $q$-Euler numbers.

In this paper, we obtain the generating function of Carlitz's type $q$-Euler numbers in the $p$-adic case.
Also, we give Witt's formula for Carlitz's type $q$-Euler numbers,
which is a partial answer to the problem in \cite{KKL}.
Moreover, we obtain a new $p$-adic $q$-$l$-function $l_{p,q}(s,\chi)$ for Dirchlet's character $\chi,$
with the property that
$$l_{p,q}(-n,\chi)=E_{n,\chi_n,q}-\chi_n(p)[p]_q^nE_{n,\chi_n,q^p}$$
for $n\in \mathbb Z^+$
using the fermionic $p$-adic integral on $\mathbb Z_p.$

\section{Carlitz's type $q$-Euler numbers in the $p$-adic case}

Let $UD(\mathbb Z_p)$ be the space of uniformly differentiable function on $\mathbb Z_p.$
Then the $p$-adic $q$-integral of a function $f\in UD(\mathbb Z_p)$ on $\mathbb Z_p$
is defined by
\begin{equation}\label{Iqf}
I_q(f)=\int_{\Z} f(a)d\mu_q(a)=\lim_{N\rightarrow\infty}\frac1{[p^{N}]_q}
\sum_{a=0}^{p^N-1}f(a)q^a
\end{equation}
(cf. \cite{CSK,KT1,KT,KT2,KT3,TK07,K07,K08,K08-1,K09,K09-1,K09-2,TK10,KJKR,OS,OSRC}).
The bosonic $p$-adic integral on $\mathbb Z_p$ is considered as the limit $q\to1,$ i.e.,
\begin{equation}\label{q=1}
I_{1}(f)=\int_{\mathbb Z_p}f(a)d\mu_{1}(a).
\end{equation}
From (\ref{Iqf}), we have the fermionic $p$-adic integral on $\mathbb Z_p$ as
follows:
\begin{equation}\label{q=-1}
I_{-1}(f)=\lim_{q\to-1}I_q(f)=\int_{\mathbb Z_p}f(a)d\mu_{-1}(a).
\end{equation}
Using formula (\ref{q=-1}), we can readily derive the classical Euler polynomials, $E_{n}(x),$ namely
\begin{equation}\label{h-Euer}
2\int_{\mathbb Z_p}
e^{(x+y)t} d\mu_{-1}(y)=\frac{2e^{xt}}{e^t+1}=\sum_{n=0}^\infty E_{n}(x)\frac{t^n}{n!}.
\end{equation}
In particular when $x=0,$ $E_{n}(0)=E_n$ is well known the Euler numbers (cf. \cite{KT2,TK10,OS}).

By definition of $I_{-1}(f),$ we show that
\begin{equation}\label{-1-ft-eq}
I_{-1}(f_1)+I_{-1}(f)=2f(0),
\end{equation}
where $f_1(x)=f(x+1)$ (see \cite{KT2}). By (\ref{-1-ft-eq}) and induction, we obtain the following
fermionic $p$-adic integral equation
\begin{equation}\label{-1-ft-eq-1}
I_{-1}(f_n)+(-1)^{n-1}I_{-1}(f)=2\sum_{i=0}^{n-1}(-1)^{n-i-1}f(i),
\end{equation}
where $n=1,2,\ldots$ and $f_n(x)=f(x+n).$
From (\ref{-1-ft-eq-1}), we note that
\begin{align}
&I_{-1}(f_n)+I_{-1}(f)=2\sum_{i=0}^{n-1}(-1)^{i}f(i)\quad\text{if $n$ is odd};
\label{-1-ft-eq-1-o} \\
&I_{-1}(f_n)-I_{-1}(f)=2\sum_{i=0}^{n-1}(-1)^{i+1}f(i)\quad\text{if $n$ is even}.
 \label{-1-ft-eq-1-e}
\end{align}
For $x\in\Z$ and any integer $i\geq0,$ we define
\begin{equation}\label{p-binom}
\binom xi=
\begin{cases}
\frac{x(x-1)\cdots(x-i+1)}{i!}&\text{if }i\geq1, \\
1, &\text{if }i=0.
\end{cases}
\end{equation}
It is easy to see that $\binom xi\in\Z$ (see \cite[p.\,172]{Ro}).
We put $x\in\mathbb C_p$ with $\ord(x)>1/(p-1)$ and $|1-q|_p<1.$
We define $q^x$ for $x\in\Z$ by
\begin{equation}\label{q^x}
q^x=\sum_{i=0}^\infty\binom xi(q-1)^i\quad\text{and}\quad [x]_q=\sum_{i=1}^\infty\binom xi(q-1)^{i-1}.
\end{equation}
If we set $f(x)=q^x$ in (\ref{-1-ft-eq-1-o}) and (\ref{-1-ft-eq-1-e}), we have
\begin{align}
&I_{-1}(q^x)=\frac{2}{q^n+1}\sum_{i=0}^{n-1}(-1)^{i}q^i=\frac2{q+1}\quad\text{if $n$ is odd};
\label{q-ex-o} \\
&I_{-1}(q^x)=\frac{2}{q^n-1}\sum_{i=0}^{n-1}(-1)^{i+1}q^{i}=\frac2{q+1}\quad\text{if $n$ is even}.
 \label{q-ex-e}
\end{align}
Thus for each $l\in\mathbb N$ we obtain $I_{-1}(q^{lx})=\frac2{q^l+1}.$
Therefore we have
\begin{equation}\label{q-num}
\begin{aligned}
I_{-1}(q^x[x]_q^n)&=\frac1{(1-q)^n}\sum_{l=0}^{n}\binom nl(-l)^lI_{-1}(q^{(l+1)x})\\
&=\frac1{(1-q)^n}\sum_{l=0}^{n}\binom nl(-l)^l\frac2{q^{l+1}+1}.
\end{aligned}
\end{equation}
Also, if $f(x)=q^{lx}$ in (\ref{-1-ft-eq}), then
\begin{equation}\label{q-ft-0}
I_{-1}(q^{l(x+1)})+I_{-1}(q^{lx})=2f(0)=2.
\end{equation}
On the other hand, by (\ref{q-ft-0}), we obtain that
\begin{equation}\label{q-num-m}
\begin{aligned}
&I_{-1}(q^{x+1}[x+1]_q^n)+I_{-1}(q^x[x]_q^n) \\
&=\frac1{(1-q)^n}\sum_{l=0}^{n}\binom nl(-1)^l\left(I_{-1}((q^{l+1})^{x+1})+I_{-1}((q^{l+1})^x)\right)\\
&=\frac2{(1-q)^n}\sum_{l=0}^{n}\binom nl(-1)^l \\
&=0
\end{aligned}
\end{equation}
is equivalent to
\begin{equation}\label{q-num-q-int}
\begin{aligned}
0&=I_{-1}(q^{x+1}[x+1]_q^n)+I_{-1}(q^x[x]_q^n) \\
&=qI_{-1}(q^{x}(1+q[x]^n))+I_{-1}(q^x[x]_q^n) \\
&=qI_{-1}\left(q^{x}\sum_{l=0}^n\binom nlq^l[x]^l\right)+I_{-1}(q^x[x]_q^n) \\
&=q\sum_{l=0}^n\binom nlq^lI_{-1}(q^x[x]^l)+I_{-1}(q^x[x]_q^n).
\end{aligned}
\end{equation}
From the definition of fermionic $p$-adic integral on $\mathbb Z_p$
and (\ref{q-num}), we can derive the following formula
\begin{equation}\label{re-f}
\begin{aligned}
I_{-1}(q^x[x]_q^n)&=\sum_{n=0}^\infty\int_{\Z}[x]_q^nq^xd\mu_{-1}(x) \\
&= \lim_{N\to\infty}\sum_{a=0}^{p^N-1}
\frac1{(1-q)^n}\sum_{i=0}^n\binom ni(-1)^iq^{ia}(-q)^{a} \\
&= \frac1{(1-q)^n}\sum_{i=0}^n\binom ni(-1)^i\lim_{N\to\infty}\sum_{a=0}^{p^N-1}(-1)^a(q^{i+1})^a \\
&=\frac1{(1-q)^n}\sum_{i=0}^n\binom ni(-1)^i\frac{2}{1+q^{i+1}}
\end{aligned}
\end{equation}
is equivalent to
\begin{equation}\label{re-f-eq}
\begin{aligned}
\sum_{n=0}^\infty I_{-1}(q^x[x]_q^n)\frac{t^n}{n!}
&=\sum_{n=0}^\infty \frac1{(1-q)^n}\sum_{i=0}^n\binom ni(-1)^i\frac{2}{1+q^{i+1}}\frac{t^n}{n!} \\
&=2\sum_{n=0}^\infty(-q)^ne^{[n]_qt}.
\end{aligned}
\end{equation}
From (\ref{q-ft-0}), (\ref{q-num-m}), (\ref{q-num-q-int}), (\ref{re-f}) and (\ref{re-f-eq}), it is easy to show that
\begin{equation}\label{Ca-q-Eu-recu}
q\sum_{l=0}^n\binom nlq^lE_{l,q}+E_{n,q}=
\begin{cases}
2&\text{if }n=0 \\
0&\text{if } n>0,
\end{cases}
\end{equation}
where $E_{n,q}$ are Carlitz's type $q$-Euler numbers defined by (see \cite{KKL})
\begin{equation}\label{Ca-q-Eu}
F_{q}(t)=2\sum_{n=0}^\infty(-q)^ne^{[n]_qt}=\sum_{n=0}^\infty E_{n,q}\frac{t^n}{n!}.
\end{equation}
Therefore, we obtain the recurrence formula for the Carlitz's type $q$-Euler numbers as follows:
\begin{equation}\label{Ca-q-Eu-recu-bi}
q(qE+1)^n+E_{n,q}=
\begin{cases}
2&\text{if }n=0 \\
0&\text{if } n>0
\end{cases}
\end{equation}
with the usual convention of replacing $E^n$ by $E_{n,q}.$
Therefore, by (\ref{re-f-eq}), (\ref{Ca-q-Eu}) and (\ref{Ca-q-Eu-recu-bi}),
we obtain the following theorem.

\begin{theorem}[Witt's formula for $E_{n,q}$]\label{witt-open}
For $n\in\mathbb Z^+,$
$$E_{n,q}=\frac1{(1-q)^n}\sum_{i=0}^n\binom ni(-1)^i\frac{2}{1+q^{i+1}}=\int_{\Z}[x]_q^nq^xd\mu_{-1}(x),$$
which is a partial answer to the problem in \cite{KKL}.
Carlitz's type $q$-Euler numbers $E_n=E_{n,q}$ can be determined inductively by
\begin{equation}
q(qE+1)^n+E_{n,q}=
\begin{cases}
2&\text{if }n=0 \\
0&\text{if } n>0
\end{cases}
\end{equation}
with the usual convention of replacing $E^n$ by $E_{n,q}.$
\end{theorem}

Carlitz type $q$-Euler polynomials $E_{n,q}(x)$ are defined by means of the generating
function $F_{q}(x,t)$ as follows:
\begin{equation}\label{tw-q-E-gen}
F_{q}(x,t)=2\sum_{k=0}^\infty(-1)^kq^ke^{[k+x]_qt}=\sum_{n=0}^\infty E_{n,q}(x)\frac{t^n}{n!}.
\end{equation}
In the cases $x=0,$ $E_{n,q}(0)=E_{n,q}$ will be called Carlitz type $q$-Euler numbers
(cf. \cite{KT3,OS}).
We also can see that the generating functions $F_{q}(x,t)$ are determined as solutions of the
following $q$-difference equation:
\begin{equation}\label{q-diff}
F_{q}(x,t)=2e^{[x]_qt}- q e^tF_{q}(x,qt).
\end{equation}
From (\ref{tw-q-E-gen}), we get the following:

\begin{lemma}\label{gen-le}
\begin{enumerate}
\item $F_{q}(x,t)=2e^{\frac{t}{1-q}}\sum_{j=0}^\infty\left(\frac1{q-1}\right)^jq^{xj}
\frac1{1+ q^{j+1}}\frac{t^j}{j!}.$
\item $E_{n,q}(x)=2\sum_{k=0}^\infty (-1)^kq^k[k+x]_q^n.$
\end{enumerate}
\end{lemma}

It is clear from (1) and (2) of Lemma \ref{gen-le} that
\begin{equation}\label{q-exi-eq}
E_{n,q}(x)=\frac{2}{(1-q)^n}\sum_{k=0}^n\binom nk \frac{(-1)^k}{1+ q^{k+1}}q^{xk}
\end{equation}
and
\begin{equation}\label{q-sums}
\begin{aligned}
\sum_{k=0}^{m-1}(-1)^kq^k[k+x]_q^n
&=\sum_{k=0}^{\infty}(-1)^{k}q^k[k+x]_q^n \\
&\quad\quad\quad-\sum_{k=0}^{\infty}(-1)^{k+m}q^{k+m}[k+m+x]_q^n \\
&=\frac1{2}\left(E_{n,q}(x)+(-1)^{m+1}q^m E_{n,q}(x+m)\right).
\end{aligned}
\end{equation}
From (\ref{q-exi-eq}) and (\ref{q-sums}), we may state

\begin{proposition}\label{q-E-re}
If $m\in\mathbb N$ and $n\in\mathbb Z^+,$ then
\begin{enumerate}
\item $E_{n,q}(x)=\frac{2}{(1-q)^n}\sum_{k=0}^n\binom nk \frac{(-1)^k}{1+ q^{k+1}}q^{xk}.$
\item $\sum_{k=0}^{m-1}(-1)^kq^k[k+x]_q^n
=\frac1{2}\left(E_{n,q}(x)+(-1)^{m+1}q^m E_{n,q}(x+m)\right).$
\end{enumerate}
\end{proposition}

\begin{proposition}\label{q-e-witt}
For $n\in\mathbb Z^+,$ the value of $\int_{\mathbb Z_p}[x+y]_q^nq^yd\mu_{-1}(y)$ is $n!$
times the coefficient of $t^n$ in the formal expansion of
$2\sum_{k=0}^\infty(-1)^kq^ke^{[k+x]_qt}$ in powers of $t.$
That is,
$E_{n,q}(x)=\int_{\mathbb Z_p}[x+y]_q^nq^yd\mu_{-1}(y).$
\end{proposition}
\begin{proof}
From (\ref{q=-1}), we have the relation
$$\int_{\mathbb Z_p}q^{k(x+y)}q^{y}d\mu_{-1}(y)={q^{xk}}\lim_{N\rightarrow\infty}
\sum_{a=0}^{p^N-1}(-q^{k+1})^a
=\frac{2q^{xk}}{1+ q^{k+1}}$$
which leads to
$$\begin{aligned}
\int_{\mathbb Z_p}[x+y]_q^nq^yd\mu_{-1}(y)
&=2\sum_{k=0}^n\binom nk\frac1{(1-q)^n}(-1)^k\int_{\mathbb Z_p}q^{k(x+y)}q^yd\mu_{-1}(y) \\
&=\frac{2}{(1-q)^n}\sum_{k=0}^n\binom nk \frac{(-1)^k}{1+ q^{k+1}}q^{xk}.
\end{aligned}$$
The result now follows by using (1) of Proposition \ref{q-E-re}.
\end{proof}

\begin{corollary}\label{q-E-re-co}
If $n\in\mathbb Z^+,$ then
$$E_{n,q}(x)=\sum_{k=0}^n\binom nk [x]_q^{n-k}q^{kx}E_{k,q}.$$
\end{corollary}

Let $d\in\mathbb N$ with $d\equiv1\pmod{2}$ and $p$ be a fixed odd prime number.
We set
\begin{equation}\label{inv}
\begin{aligned}
&X=\varprojlim_N (\mathbb Z/dp^N\mathbb Z), \quad
X^*=\bigcup_{\substack{ 0<a<dp\\ (a,p)=1}} a+dp\mathbb Z_p,\\
&\,a+dp^N\mathbb Z_p=\{x\in X \mid x\equiv a\pmod{dp^N}\},
\end{aligned}
\end{equation}
where $a\in \mathbb Z$ with $0\leq a<dp^N$ (cf. \cite{KT2,TK07}).
Note that the natural map $\mathbb Z/dp^N\mathbb Z\to\mathbb Z/p^N\mathbb Z$
induces
\begin{equation}\label{pro}
\pi:X\to\Z.
\end{equation}
Hereafter, if $f$ is a function on $\Z,$ we denote by the same $f$ the function $f\circ\pi$ on $X.$
Namely we consider $f$ as a function on $X.$

Let $\chi$ be the Dirichlet's character with an odd conductor $d=d_\chi\in\mathbb N.$
Then the generalized Carlitz type $q$-Euler polynomials attached to $\chi$ defined by
\begin{equation}\label{open}
E_{n,\chi,q}(x)=\int_{X}\chi(a)[x+y]_q^nq^yd\mu_{-1}(y),
\end{equation}
where $n\in\mathbb Z^+$ and $x\in\mathbb Z_p.$
Then we have the generating function of generalized Carlitz type $q$-Euler polynomials attached to $\chi:$
\begin{equation}\label{prob-1}
F_{q,\chi}(x,t)=2\sum_{m=0}^\infty\chi(m)(-1)^mq^me^{[m+x]_qt} =\sum_{n=0}^\infty E_{n,\chi,q}(x)\frac{t^n}{n!}.
\end{equation}
Now fixed any $t\in\mathbb C_p$ with $\ord(t)>1/(p-1)$ and $|1-q|_p<1.$
From (\ref{prob-1}), we have
\begin{equation}\label{prob-2}
\begin{aligned}
F_{q,\chi}(x,t)&=2\sum_{m=0}^\infty\chi(m)(-q)^m\sum_{n=0}^\infty\frac1{(1-q)^n}\sum_{i=0}^n\binom ni(-1)^iq^{i(m+x)}\frac{t^n}{n!} \\
&=2\sum_{n=0}^\infty\frac1{(1-q)^n}\sum_{i=0}^n\binom ni(-1)^iq^{ix} \\
&\quad\times
\sum_{j=0}^{d-1}\sum_{l=0}^\infty\chi(j+dl)(-q)^{j+dl}q^{i(j+dl)}\frac{t^n}{n!} \\
&=2\sum_{n=0}^\infty\frac1{(1-q)^n}\sum_{j=0}^{d-1}\chi(j)(-q)^j
\sum_{i=0}^n\binom ni(-1)^i\frac{q^{i(x+j)}}{1+q^{d(i+1)}}\frac{t^n}{n!},
\end{aligned}
\end{equation}
where $x\in\mathbb Z_p$ and $d\in\mathbb N$ with $d\equiv1\pmod2.$
By (\ref{prob-1}) and (\ref{prob-2}), we can derive the following formula
\begin{equation}\label{prob-3}
\begin{aligned}
E_{n,\chi,q}(x)&=\frac1{(1-q)^n}\sum_{j=0}^{d-1}\chi(j)(-q)^j
\sum_{i=0}^n\binom ni(-1)^iq^{i(x+j)} \frac{2}{1+q^{d(i+1)}} \\
&=\frac1{(1-q)^n}\sum_{j=0}^{d-1}\chi(j)(-q)^j
\sum_{i=0}^n\binom ni(-1)^iq^{i(x+j)} \\
&\quad\times\lim_{N\to\infty}\sum_{l=0}^{p^N-1}(-1)^l(q^{d(i+1)})^l \\
&=\lim_{N\to\infty}\sum_{j=0}^{d-1}\sum_{l=0}^{p^N-1}\chi(j+dl)
\frac1{(1-q)^n}\sum_{i=0}^n\binom ni(-1)^iq^{i(j+dl+x)} \\
&\quad\times(-1)^{j+dl}q^{j+dl} \\
&=\lim_{N\to\infty}\sum_{a=0}^{dp^N-1}\chi(a)
\frac1{(1-q)^n}\sum_{i=0}^n\binom ni(-1)^iq^{i(a+x)}(-q)^{a} \\
&=\int_{X}\chi(y)[x+y]_q^nq^yd\mu_{-1}(y),
\end{aligned}
\end{equation}
where $x\in\mathbb Z_p$ and $d\in\mathbb N$ with $d\equiv1\pmod2.$
Therefore, we obtain the following

\begin{theorem}
$$
E_{n,\chi,q}(x)=\frac1{(1-q)^n}\sum_{j=0}^{d-1}\chi(j)(-q)^j
\sum_{i=0}^n\binom ni(-1)^iq^{i(x+j)} \frac{2}{1+q^{d(i+1)}},
$$
where $n\in\mathbb Z^+$ and $x\in\mathbb Z_p.$
\end{theorem}

Let $\omega$ denote the Teichm\"uller character mod $p.$ For $x\in X^*,$ we set
\begin{equation}\label{Tei}
\langle x\rangle=[x]_q\omega^{-1}(x)=\frac{[x]_q}{\omega(x)}.
\end{equation}
Note that since $|\langle x\rangle-1|_p<p^{-1/(p-1)},$ $\langle x\rangle^s$ is defined
by $\exp(s\log_p\langle x\rangle)$ for $|s|_p\leq1$ (cf. \cite{K07,K08-1,Sa}).
We note that $\langle x\rangle^s$ is analytic for $s\in\Z.$

We define an interpolation function for Carlitz type $q$-Euler numbers.
For $s\in\Z,$
\begin{equation}\label{p-l-q}
l_{p,q}(s,\chi)=\int_{X^*}\langle x\rangle^{-s}\chi(x)q^xd\mu_{-1}(x).
\end{equation}
Then $l_{p,q}(s,\chi)$ is analytic for $s\in\Z.$

The values of this function at non-positive integers are given by

\begin{theorem}\label{int-p-ell}
For integers $n\geq0,$
$$l_{p,q}(-n,\chi)=E_{n,\chi_n,q}-\chi_n(p)[p]_q^nE_{n,\chi_n,q^p},$$
where $\chi_n=\chi\omega^{-n}.$ In particular, if $\chi=\omega^n,$ then
$l_{p,q}(-n,\omega^n)=E_{n,q}-[p]_qE_{n,q^p}.$
\end{theorem}
\begin{proof}
$$
\begin{aligned}
l_{p,q}(-n,\chi)&=\int_{X^*}\langle x\rangle^{n}\chi(x)q^xd\mu_{-1}(x) \\
&=\int_{X}[x]_q^n\chi_n(x)q^xd\mu_{-1}(x)-\int_{X}[px]_q^n\chi_n(px)q^{px}d\mu_{-1}(px) \\
&=\int_{X}[x]_q^n\chi_n(x)q^xd\mu_{-1}(x)-[p]_q^n\chi_n(p)\int_{X}[x]_{q^p}^n\chi_n(x)q^{px}d\mu_{-1}(x).
\end{aligned}
$$
Therefore by (\ref{open}), the theorem is proved.
\end{proof}

Let $\chi$ be the Dirichlet's character with an odd conductor $d=d_\chi\in\mathbb N.$
Let $F$ be a positive integer multiple of $p$ and $d.$ Then by (\ref{tw-q-E-gen}) and (\ref{prob-1}), we have
\begin{equation}\label{rabbe}
\begin{aligned}
F_{q,\chi}(x,t)&=2\sum_{m=0}^\infty\chi(m)(-1)^mq^me^{[m+x]_qt} \\
&=2\sum_{a=0}^{F-1}\chi(a)(-q)^a\sum_{k=0}^\infty(-q)^{Fk}e^{[F]_q\left[k+\frac{x+a}{F}\right]_{q^F}t} \\
&=\sum_{n=0}^\infty\left([F]_q^n\sum_{a=0}^{F-1}\chi(a)(-q)^aE_{n,q^F}\left(\frac{x+a}{F}\right)\right)\frac{t^n}{n!}.
\end{aligned}
\end{equation}
Therefore we obtain the following
\begin{equation}\label{rabbe-lem}
E_{n,\chi,q}(x)=[F]_q^n\sum_{a=0}^{F-1}\chi(a)(-q)^aE_{n,q^F}\left(\frac{x+a}{F}\right).
\end{equation}
If $\chi_n(p)\neq0,$ then $(p,d_{\chi_n})=1,$ so that $F/p$ is a multiple of $d_{\chi_n}.$
From (\ref{rabbe-lem}), we derive
\begin{equation}\label{rabbe-ell}
\begin{aligned}
\chi_n(p)[p]_q^nE_{n,\chi_n,q^p}
&=\chi_n(p)[p]_q^n[F/p]_{q^p}^n\sum_{a=0}^{F/p-1}\chi_n(a)(-q^p)^aE_{n,(q^p)^{F/p}}\left(\frac{a}{F/p}\right)\\
&=[F]_q^n\sum_{\substack{a=0\\p\mid a}}^F\chi_n(a)(-q)^aE_{n,q^F}\left(\frac{a}{F}\right).
\end{aligned}
\end{equation}
Thus we have
\begin{equation}\label{ell-le}
\begin{aligned}
E_{n,\chi_n,q}-\chi_n(p)[p]_q^nE_{n,\chi_n,q^p}
&=[F]_q^n\sum_{\substack{a=0\\p\nmid a}}^{F-1}\chi_n(a)(-q)^aE_{n,q^F}\left(\frac{a}{F}\right).
\end{aligned}
\end{equation}
By Corollary \ref{q-E-re-co}, we easily see that
\begin{equation}\label{dis-poly}
\begin{aligned}
E_{n,q^F}\left(\frac{a}{F}\right)&=\sum_{k=0}^n\binom nk \left[\frac{a}{F}\right]_{q^F}^{n-k}q^{ka}E_{k,q^F} \\
&=[F]_q^{-n}[a]_q^n\sum_{k=0}^n\binom nk \left[\frac{F}{a}\right]_{q^a}^{k}q^{ka}E_{k,q^F}.
\end{aligned}
\end{equation}
From (\ref{ell-le}) and (\ref{dis-poly}), we have
\begin{equation}\label{ell-le-ft}
\begin{aligned}
E_{n,\chi_n,q}&-\chi_n(p)[p]_q^nE_{n,\chi_n,q^p} \\
&=[F]_q^n\sum_{\substack{a=0\\p\nmid a}}^{F-1}\chi_n(a)(-q)^aE_{n,q^F}\left(\frac{a}{F}\right)\\
&=\sum_{\substack{a=0\\p\nmid a}}^{F-1}\chi(a)\langle a\rangle^n(-q)^a
\sum_{k=0}^\infty\binom nk \left[\frac{F}{a}\right]_{q^a}^{k}q^{ka}E_{k,q^F},
\end{aligned}
\end{equation}
since $\chi_n(a)=\chi(a)\omega^{-n}(a).$ From Theorem \ref{int-p-ell} and (\ref{ell-le-ft}),
\begin{equation}\label{ne-val}
\begin{aligned}
l_{p,q}(-n,\chi)=\sum_{\substack{a=0\\p\nmid a}}^{F-1}\chi(a)\langle a\rangle^n(-q)^a
\sum_{k=0}^\infty\binom nk \left[\frac{F}{a}\right]_{q^a}^{k}q^{ka}E_{k,q^F}
\end{aligned}
\end{equation}
for $n\in\mathbb Z^+.$
Therefore we have the following theorem.

\begin{theorem}
Let $F$ be a positive integer multiple of $p$ and $d=d_\chi,$ and let
$$
l_{p,q}(s,\chi)=\int_{X^*}\langle x\rangle^{-s}\chi(x)q^xd\mu_{-1}(x), \quad s\in\Z.
$$
Then $l_{p,q}(s,\chi)$ is analytic for $s\in\Z$ and
$$l_{p,q}(s,\chi)=\sum_{\substack{a=0\\p\nmid a}}^{F-1}\chi(a)\langle a\rangle^{-s}(-q)^a
\sum_{k=0}^\infty\binom{-s}k \left[\frac{F}{a}\right]_{q^a}^{k}q^{ka}E_{k,q^F}.$$
Furthermore, for $n\in\mathbb Z^+$
$$l_{p,q}(-n,\chi)=E_{n,\chi_n,q}-\chi_n(p)[p]_q^nE_{n,\chi_n,q^p}.$$
\end{theorem}

%\section*{Acknowledgment}
%We thank the anonymous referee for the many useful comments, which have improved our exposition.

\bibliography{central}

\end{document}